\documentclass[10pt,final,twocolumn]{IEEEtran}

\IEEEoverridecommandlockouts                        
 \usepackage{graphicx}

\usepackage{wrapfig} 
\usepackage{times} 
\usepackage{amsmath,verbatim}
\usepackage{amsmath}
\usepackage{subfig}

\usepackage{amssymb}  
\usepackage{dsfont} 
\usepackage[dvipsnames,usenames]{color}
\usepackage[colorlinks,%
linkcolor=BrickRed,%
filecolor =BrickRed,%
citecolor=RoyalPurple,%
]{hyperref}

\usepackage{amsmath} 
\usepackage{amssymb}  
\usepackage{color}
\usepackage{cite}

\newtheorem{theorem}{Theorem}
\newtheorem{lemma}{Lemma}
\newtheorem{definition}{Definition}
\newtheorem{assumption}{Assumption}
\newtheorem{proposition}{Proposition}

\newcommand{\an}[1]{{\color{black}#1}}
\newcommand{\fy}[1]{{\color{black}#1}}
\newcommand{\us}[1]{{\color{black}#1}}

 \newcommand{\remove}[1]{}
\newcommand{\EXP}[1]{\mathsf{E}\!\left[#1\right] }

\def\sF{\mathcal{F}}

\def\Real{\mathbb{R}}
\def\g{\gamma}

\def\e{\epsilon}
\def\a{\alpha}

\def\us#1{{{\color{black}#1}}}

\usepackage{ulem}
\def\us#1{{\color{black}{#1}}}
\def\fy#1{{\color{black}{#1}}}

\author{Farzad~Yousefian,   
        Angelia~Nedi\'c, and   
		Uday V.~Shanbhag \thanks{The first two authors are in the Dept.
			of Indust. and Enterprise
Sys. Engg., Univ. of Illinois, Urbana, IL 61801, USA,
		while the last author is with the Department of Indust. and
			Manuf. Engg., Penn. State University,
						  University Park, PA 16802, USA. They are
							  contactable at 
{\tt\small \{yousefi1,angelia\}@illinois.edu} and {\tt \small
	udaybag@psu.edu}. Nedi\'{c} and
Shanbhag gratefully acknowledge the support of the NSF through awards CCF-0904619 and ONR grant No. 00014-12-1-0998 (Nedi\'{c}) and 
CMMI-1246887 (Shanbhag).}}

\title{\LARGE \bf Optimal robust smoothing extragradient algorithms \\ for stochastic variational inequality problems}
\begin{document}
\maketitle
\thispagestyle{empty}
\pagestyle{plain}
\vspace{-0.5in}
\begin{abstract}
\fy{We consider stochastic variational inequality problems where the mapping
is monotone over a compact convex set. We present two robust variants of 
stochastic extragradient algorithms for solving such problems. Of these, the first scheme employs an iterative averaging technique where we consider a generalized
choice for the weights in the averaged sequence. Our first contribution
is to show that using an appropriate choice for these weights, a suitably defined gap function
attains the optimal rate of convergence 
${\cal O}\left(\frac{1}{\sqrt{k}}\right)$. In the second part of the paper,
	under an additional assumption of
	weak-sharpness, we update the stepsize sequence using a recursive rule that leverages problem parameters. The second contribution lies in
	showing that employing such a sequence, the extragradient
	algorithm possesses almost-sure convergence to the solution as well
	as convergence in a mean-squared sense to the solution of the problem at the rate
	${\cal O}\left(\frac{1}{k}\right)$. Motivated by the absence of a Lipschitzian parameter, in both schemes we utilize a locally randomized smoothing scheme. Importantly, by approximating a smooth mapping, this scheme enables us to estimate the Lipschitzian parameter.  The smoothing parameter is updated per iteration and we show convergence to the solution of the original problem in both algorithms. }
\end{abstract} 
\maketitle

\section{Introduction}\label{sec:introduction}
The theory of variational inequality (VI) was introduced in mid-1960s, motivated by the elastostatic equilibrium problems. During the past five decades, this subject has been a powerful framework in modeling a wide range of optimization and equilibrium problems in operations research, engineering, finance, and economics (cf. \cite{facchinei02finite}, \cite{Rockafellar98}). Given a set $X \subset
\Real^n$ and a mapping $F:X \rightarrow \mathbb{R}^n$, a {VI} problem, denoted
by VI$(X,F)$, requires \an{determining} an $x^* \in X$ 
such that $F(x^*)^T(x-x^*)\geq 0$ for any $x \in X$. In this paper, our interest lies in computation of solutions to VI problems with uncertain settings. We
consider the case  \us{where $F:X \rightarrow 
			   \mathbb{R}^n$ represents the expected value of a
			   stochastic mapping ${\Phi}:X\times\Omega \rightarrow
			   \mathbb{R}^n$}, i.e., $F_i(x)\triangleq
{\EXP{\Phi_i(x,\xi{(\omega)})}}$ \us{where $\xi : \Omega \to \Real^d$ is a
$d-$dimensional random
variable and $(\Omega, {\cal F}, \mathbb{P})$ denotes the associated
probability space}. Consequently, $x^*
\in X$ solves VI$(X,F)$ if 
\begin{align}\label{def:SVI}  
\hspace{-0.15in}{\EXP{\Phi(x^*,\xi{(\omega)})}^T(x-x^*)} \geq 0, \, \hbox{for
	\us{every} } x \in X. 
\end{align}
The stochastic VI problem (\ref{def:SVI}) arises in many situations,
\an{often} 
modeling stochastic convex optimization problems and stochastic Nash \fy{equilibrium} problems.
Utilized by sampling from an unbiased stochastic oracle $\Phi(x,\xi)$, iterative algorithms have been developed to solve the problem (\ref{def:SVI}). Such schemes include the stochastic approximation type methods \cite{nemirovski_robust_2009}, \cite{Farzad-WSC13} and extragradient type methods \cite{Nem11},\cite{Lan-VI-13}, \cite{Nem04}. In a recent work \cite{Aswin-ACC14}, Kannan and Shanbhag studied almost-sure convergence of extragradient algorithms and provided sufficiency conditions for the solvability of stochastic VIs with pseudo-monotone mappings.  Prox-type methods were first developed by
Nemirovski  \cite{Nem04} for solving VIs with monotone and Lipschitz
mappings and addressed different problems in convex optimization and
variational inequalities.
Recently, Juditsky et al. \cite{Nem11} introduced the stochastic Mirror-
Prox (SMP) algorithm for solving stochastic VIs in both smooth and nonsmooth cases. In \cite{Nem11}, it is assumed that the mapping $F$ is monotone and satisfies the following relation:  
\[\|F(x)-F(y)\| \leq L\|x-y\|+B, \quad \hbox{for all }x,y \in X,\] where $L\geq 0$ and $B\geq 0$ are known constants. Under such conditions, by choosing a constant stepsize rule $0< \g <\frac{1}{\sqrt{3L}}$, the optimal rate of convergence of a suitably defined gap function is shown to be $O(1)\left(\frac{L}{t}+\frac{B+\sigma}{\sqrt{t}}\right)$ where $\sigma$ is the upper bound on the variance of the stochastic oracle and $t$ is the pre-fixed number of iterations. In this paper, our main goal is developing two classes of robust extragradient algorithms for monotone stochastic VIs in extension of the work in \cite{Nem11} and \cite{Aswin-ACC14}. The first class of the proposed \fy{extragradient }algorithms employs the well-known averaging technique utilized by new choices of the averaging weights. In the second part of the paper, we consider monotone VI problems with weak-sharpness property and  we develop an extragradient algorithm that employs a recursive stepsize policy. Such a stepsize sequence is \fy{obtained} in terms of problem parameters. The word ``robust'' refers to the self-tuned stepsize rule and capability of dealing with the absence of a Lipschitz constant. Our main contributions in this paper are described as follows: \\ 
(1) \textit{Choice of the averaging weights:} The SMP algorithm in \cite{Nem11} generates a wighted-iterative averaging sequence of the form $\bar x_t \triangleq \sum_{t=1}^k\frac{\g_t}{\sum_{t=0}^k\g_t}x_t$ where $x_t$ is generated at the t-$th$ iteration and $\g_t>0$ is the corresponding stepsize. Recently, Nedi\'c and Lee \cite{Nedic14} showed that using different weights of the form $\frac{\g_t^{-1}}{\sum_{t=0}^k\g_t^{-1}}$, the subgradient mirror-descent algorithms attain the optimal rate of convergence without requiring window-based averaging sequences \fy{similar to} \cite{Ghad12} and \cite{nemirovski_robust_2009}. In this paper, we generalize this idea in two directions: First, we show that such choices can be applied in the stochastic extragradient algorithms (e.g. \cite{Nem11}). Second, using the weights $\frac{\g_t^{r}}{\sum_{t=0}^k\g_t^{r}}$ where $r \in \Real$ is a constant, we show that for any $r<1$, the optimal convergence rate is attained. Note that this optimal rate cannot be attained when $r=1$ (e.g. in \cite{Nem11}). 
\\ 
(2) \textit{Developing parameterized stepsizes:} In the second part of the paper, we assume that the problem with monotone mapping has a weak-sharpness property. We develop a recursive stepsize sequence that leverages problem parameters and show that by employing such a stepsize rule, the sequence $\{x_t\}$ generated by the extragradient algorithm converges in almost-sure sense to the solution of the problem. Moreover, we show that this robust scheme converges in mean-squared sense to the solution of the problem with the rate of ${\cal O}\left(\frac{1}{k}\right)$.\\
(3) \textit{Estimating the Lipschitzian parameter \fy{$L$}:} While the SMP algorithm in \cite{Nem11} addresses both smooth and nonsmooth problems by allowing $L$ to be zero, knowing the constant $L$ benefits the rate of convergence since the term $\frac{L}{t}$ decays faster than $\frac{B+\sigma}{\sqrt{t}}$. Moreover, the prescribed constant stepsize in \cite{Nem11} is bounded in terms of $L$. We consider the case that the mapping is either nonsmooth or the constant $L$ is unknown. Motivated by a smoothing technique first introduced by Steklov \cite{steklov1}, our goal is addressing these cases.
\fy{Recently} in \cite{Farzad1} and \cite{Farzad-WSC13}, by employing this technique
we addressed nonsmoothness in developing adaptive stepsizes stochastic approximation schemes in absence or unavailability of a Lipschitz constant. We extend those results by applying such smoothing technique for both extragradient \fy{schemes}.

The paper is organized as follows: in Section~\ref{sec:algorithmI}, 
\an{we present an algorithm utilizing
the weighted averaging} and we show the convergence of a suitably defined gap function to zero with the optimal rate. In Section~\ref{sec:algorithmII}
\an{we present an algorithm with a recursive stepsize updates and provide its} convergence and rate analysis. 
\an{We conclude with some remarks} in Section~\ref{sec:conc}.

\textbf{Notation:} In this paper, a vector $x$ is assumed to be
a column vector, $x^T$ denotes the transpose of a vector $x$, and
$\|x\|$ denotes the Euclidean vector norm, i.e., $\|x\|=\sqrt{x^Tx}$.
We let $\Pi_X(x)$ denote the Euclidean projection of a vector $x$ on
a set $X$, i.e., $\|x-\Pi_X(x)\|=\min_{y \in X}\|x-y\|$. We write \textit{a.s.} as the abbreviation for ``almost
surely''. We use $\EXP{z}$ to denote the expectation of a random variable~$z$.

\section{Stochastic Extragradient Method}\label{sec:algorithmI}
\us{We consider the following stochastic variant of 
\an{Korpelevich's
	extragradient scheme:} \an{for all $t\ge0$:}
	}
\begin{align}\label{algorithm:RSEG}
\begin{aligned}
y_{t+1}&=\Pi_{X}\left[x_t-\g_t\Phi(x_t+z_t',\xi_t')\right],\\
x_{t+1}&=\Pi_{X}\left[x_t-\g_t\Phi(y_{t+1}+z_t,\xi_t)\right].
\end{aligned}
\end{align}
\an{In this scheme, $\{\g_t\}$ is the stepsize sequence and $x_0\in X$ is a random initial point with $\EXP{\|x_0\|^2}\leq \infty$.} 
\an{The vectors $\xi_t$ and $\xi_t'$} are two i.i.d samples from the probability space $(\Omega, {\cal F}, \mathbb{P})$. 
\an{Also, $z_t$ and $z_t'$ are two i.i.d.\ samples from a uniform random variable $Z_t \in \Real^n$. The $i$th element of  $Z_t$, denoted by $Z_{t,i}$ is uniformly distributed in the interval $[-\frac{\e_t}{2},\frac{\e_t}{2}]$.} 
To have a well-defined algorithm, we define the set $X^\e \triangleq X+C_n(0,\e)$, where $C_n(0,\e)\triangleq \{x\in\Real^n\mid  x_i\in [-\frac{\e}{2},\frac{\e}{2}]\}$ is a cube centered at origin. 
\an{The scalar $\e$ is assumed to be an upper bound for the sequence $\{\e_t\}$. We denote the history of the scheme using the following notations:
\begin{align}\label{def:filteration}		
\sF_t' & \triangleq \{x_0\}\cup \{\xi_0',z_0',\xi_1',z_1',\ldots,\xi_{t}',z_{t}'\},\cr \sF_t & \triangleq
\{\xi_0,z_0,\xi_1,z_1,\ldots,\xi_{t},z_{t}\},
\end{align}
		where $t \geq 0$.
		The first set of assumptions is on the set $X$, the mapping $F$, and the random variables.
		}
\begin{assumption}\label{assum:step_error_sub_1} 
\an{Let the following hold:\\ 
(a)~The set $X \subset \Real^n$ is closed, convex, and bounded, i.e.,  
$\|x\|\leq M$ for  all  $x\in X$ and some $M>0$.
(b)~The mapping $F$ is monotone and bounded on the set $X^\e$, i.e., 
$\|F(x)\|\leq C$ for all $x\in X^\e$ and some $C>0$.
(c$)$~There exists an $x^* \in X$ such that $(x-x^*)^T\EXP{\Phi(x^*,\xi)}\geq 0$, for all $x \in X$.
(d)~The random variables $z_t$, $z_t'$, $\xi_k$ and $\xi_k'$ are 
all i.i.d.\ and independent from each other for any $t,k \geq 0$.}
\end{assumption}

\an{We also make use of the following assumptions}.
\begin{assumption}\label{assump:stochastic-error} 
Define $w(x)\triangleq \Phi(x,\xi)-F(x)$ for $x \in X^\e$. We assume, the samples $\Phi(x,\xi)$ taken in algorithm (\ref{algorithm:RSEG}) are unbiased, i.e., 
$\EXP{w(x)}=0$  for all $x \in X^\e$. Moreover, the variance of the samples $\Phi(x,\xi)$ is bounded, 
i.e., \an{there is $\nu>0$} such that $\EXP{\|w(x)\|^2}\leq \nu^2$ for all $x \in X^\e$.
\end{assumption}

\an{Also, we define the approximate mapping 
\begin{align}\label{def:F_k}
F_t(x)\triangleq \EXP{F(x+Z_t)}, \quad \hbox{for all $x \in X$ and all $t\ge0$}. 
\end{align}
\an{The following result has been shown 
in our prior work~\cite{Farzad03}} (on random local smoothing): 
\begin{lemma}\label{lemma:F_k}
Let the mapping $F_t:X\rightarrow \Real^n$ be defined by (\ref{def:F_k}) 
where $Z_t$ is uniformly distributed over $C_n(0,\e_t)$. 
Then, \an{for all $t\ge0$,} $F_t$ is Lipschitz continuous over the set $X$, i.e.,
\[\|F_t(x)-F_t(y)\|\leq \frac{\sqrt{n}C}{\e_t}\|x-y\|, \quad \hbox{for all }x,y \in X.\]
\end{lemma}
}

In our analysis, \an{we exploit} the following properties of the projection mapping (cf.~Chapter 2 in \cite{Nedic03}).
\begin{lemma}[Properties of the projection mapping]\label{lemma:projProperties}
Let \an{$X \subseteq\Real^n$} be a nonempty closed convex set. 
\begin{itemize}
\item [(a)] 
$\|\Pi_{X}[u]-\Pi_{X}[v]\| \leq \|u-v\|,$ for all $u,v \in \Real^n.$
\item [(b)] 
$(\Pi_X[u]-u)^T(x-\Pi_X[u])\geq 0$, for all $u \in \Real^n$ and $x \in X.$ 
\end{itemize}
\end{lemma}

\an{For notational convenience, we define the stochastic errors of algorithm (\ref{algorithm:RSEG}) as follows:
\begin{align}\label{equ:stochastic-errors}
  \begin{split}
 w_t&\triangleq\Phi(y_{t+1}+z_t,\xi_{t})-F(y_{t+1}+z_t), \\
 w_t'&\triangleq \Phi(x_{t}+z_t',\xi_{t}')-F(x_{t}+z_t'),\\
\Delta_t&\triangleq F(y_{t+1}+z_t)-F_t(y_{t+1}),\cr 
\Delta_t'&\triangleq F(x_{t}+z_t')-F_t(x_t).  \end{split}
\end{align}
We have the following basic result for the algorithm.}
\begin{lemma}\label{lemma:rec-bound} 
Let \an{Assumptions~\ref{assum:step_error_sub_1} and~\ref{assump:stochastic-error} hold,
and $0<\g_t\leq \frac{\e_t}{\sqrt{5n}C}$ for all $t\geq 0$.
Then, for the iterates of algorithm (\ref{algorithm:RSEG}),
the following relation holds for all $y \in X$ and all $t\ge0$:}
 \begin{align}\label{ineq:rec-bound}
& \quad \|x_{t+1}-y\|^2 \cr
&\leq \|x_t-y\|^2+2 \notag
\g_tF(y_{t+1}+z_t)^T(y-(y_{t+1}+z_t)) \\
& +2\sqrt{n}C\g_t\e_t+2\g_tw_t^T(y-y_{t+1}) +5\g_t^2B_t,
\end{align}
where $B_t\triangleq \|\Delta_t\|^2 +\|\Delta_t'\|^2 +\|w_t\|^2+\|w_t'\|^2$,
\an{and $\Delta_t$, $\Delta_t'$, $w_t$ and $w_t'$ are as defined in~\eqref{equ:stochastic-errors}.}
\end{lemma}
\begin{proof}
Let \an{$y\in X$ and $t\geq 0$ be fixed, but arbitrary}. We start by
estimating an upper bound for the stochastic term $\|x_{t+1}-y\|^2$.
From algorithm \eqref{algorithm:RSEG}, we have the following:
  \begin{align}\label{ineq:01}
& \|x_{t+1}-y\|^2 = \|x_{t+1}-x_t+x_t-y\|^2\cr
	& =\|x_{t+1}-x_t\|^2+\|x_t-y\|^2+2(x_{t+1}-x_t)^T(x_t-y)\cr
&=\|x_{t+1}-x_t\|^2+\|x_t-y\|^2\notag \\
  &
  +2(\us{x_{t+1}-x_t})^T(x_t-x_{t+1})+2(x_{t+1}-x_t)^T(x_{t+1}-y)\notag
  \\
&=\|x_t-y\|^2-\|x_{t+1}-x_t\|^2  +2(x_{t+1}-x_t)^T(x_{t+1}-y)\qquad
\end{align}
where in the second equality, we add and subtract $x_t$, and in the third equality, 
we add and subtract $x_{t+1}$. Consider Lemma~\ref{lemma:projProperties}(b), and 
let $u\triangleq x_t-\g_t\Phi(y_{t+1}+z_t,\xi_t)$ and $x\triangleq y$. 
\an{By algorithm~(\ref{algorithm:RSEG}), we have $x_{t+1}=\Pi_X[u]$. 
Thus, by Lemma~\ref{lemma:projProperties}(b), we obtain}
\begin{align*}
&0  \leq
	\left(x_{t+1}-(x_t-\g_t\Phi(y_{t+1}+z_t,\xi_t))\right)^T(y-x_{t+1})
	\\
 & = (x_{t+1}-x_t)^T(y-x_{t+1})  +\g_t\Phi(y_{t+1}+z_t,\xi_t)^T(y-x_{t+1}).
 \end{align*}
\an{Hence
$(x_{t+1}-x_t)^T(x_{t+1}-y)  \leq \g_t\Phi(y_{t+1}+z_t,\xi_t)^T(y-x_{t+1}),$ 
and by relation~(\ref{ineq:01}), it follows that}
  \begin{align*}
\|x_{t+1}-y\|^2 & \leq \|x_t-y\|^2-\|x_{t+1}-x_t\|^2\\
	& +2 \g_t\Phi(y_{t+1}+z_t,\xi_t)^T(y-x_{t+1}).
\end{align*}
\an{By adding and subtracting $y_{t+1}$ in $\|x_{t+1}-x_t\|^2$,} we have
  \begin{align}
 \|x_{t+1}-y\|^2
		&\leq \|x_t-y\|^2-\|x_{t+1}-y_{t+1}+ y_{t+1}-x_t\|^2\notag \\
		& +2 \g_t\Phi(y_{t+1}+z_t,\xi_t)^T(y-x_{t+1})\notag \\
&= \|x_t-y\|^2-\|x_{t+1}-y_{t+1}\|^2\notag \\
& - \|y_{t+1}-x_t\|^2-2(x_{t+1}-y_{t+1})^T(y_{t+1}-x_t)\notag \\
	&+2 \g_t\Phi(y_{t+1}+z_t,\xi_t)^T(y-x_{t+1}).\label{ineq:02}
\end{align}
Using Lemma~\ref{lemma:projProperties}(b) 
with $u\triangleq x_t-\g_t\Phi(x_{t}+z_t',\xi_t')$ and $x\triangleq x_{t+1}$, and \an{using $y_{t+1}=\Pi_X[u]$
(see algorithm (\ref{algorithm:RSEG})), 
we obtain}
  \begin{align*}
0 & \leq
\left(y_{t+1}-(x_t-\g_t\Phi(x_{t}+z_t',\xi_t'))\right)^T(x_{t+1}-y_{t+1})
\\
	& = (y_{t+1}-x_t)^T(x_{t+1}-y_{t+1})\cr
	&\  \ +\g_t\Phi(x_{t}+z_t',\xi_t')^T(x_{t+1}-y_{t+1}).
\end{align*}
\an{Therefore, 
$-(y_{t+1}-x_t)^T(x_{t+1}-y_{t+1}) \leq \g_t\Phi(x_{t}+z_t',\xi_t')^T(x_{t+1}-y_{t+1})$, 
which together with relation~(\ref{ineq:02}), yields}\begin{align*}
& \quad \|x_{t+1}-y\|^2 \leq  \|x_t-y\|^2-\|x_{t+1}-y_{t+1}\|^2\\
	& - \|y_{t+1}-x_t\|^2+2\g_t\Phi(x_{t}+z_t',\xi_t')^T(x_{t+1}-y_{t+1})\cr &+2 \g_t\Phi(y_{t+1}+z_t,\xi_t)^T(y-x_{t+1}).
\end{align*}
\an{By adding and subtracting $y_{t+1}$ in the appropriate terms of the preceding relation, we further have}
  \begin{align}\label{ineq:03}
& \quad \|x_{t+1}-y\|^2 \leq  \|x_t-y\|^2-\|x_{t+1}-y_{t+1}\|^2\notag\\
	& - \|y_{t+1}-x_t\|^2+2\g_t\Phi(x_{t}+z_t',\xi_t')^T(x_{t+1}-y_{t+1})\cr &+2 \g_t\Phi(y_{t+1}+z_t,\xi_t)^T(y-y_{t+1}+y_{t+1}-x_{t+1})\cr 
&=\|x_t-y\|^2\underbrace{-\|x_{t+1}-y_{t+1}\|^2}_{\tiny \mbox{Term 1}}-
\|y_{t+1}-x_t\|^2\\ \notag
	&+2 \g_t\Phi(y_{t+1}+z_t,\xi_t)^T(y-y_{t+1})\\
&\underbrace{+2\g_t\left(\Phi(x_{t}+z_t',\xi_t')-\Phi(y_{t+1}+z_t,\xi_t)\right)^T(x_{t+1}-y_{t+1})}_{\tiny
	\mbox{Term 2}}.\notag
\end{align}
From the \an{definitions} of Terms 1 and 2, \us{by employing $2ab \leq a^2 + b^2,$ for any $a,b\in\Real$,
we have that}  $$\mbox{Term 1 + Term
	2} \leq \g_t^2\|\Phi(y_{t+1}+z_t,\xi_t)-\Phi(x_{t}+z_t',\xi_t')\|^2. $$ 
	\an{The preceding inequality and relation~(\ref{ineq:03})} imply that 
 \begin{align}\label{ineq:04}
& \|x_{t+1}-y\|^2 \leq \|x_t-y\|^2- \|y_{t+1}-x_t\|^2\notag \\
	& +2 \g_t\Phi(y_{t+1}+z_t,\xi_t)^T(y-y_{t+1})\notag \\
	&+\g_t^2\|\Phi(y_{t+1}+z_t,\xi_t)-\Phi(x_{t}+z_t',\xi_t')\|^2.
\end{align}
We now estimate the term $\|\Phi(y_{t+1}+z_t,\xi_t)-\Phi(x_{t}+z_t',\xi_t')\|^2$.
\an{Using the definitions~(\ref{equ:stochastic-errors}), we have
\begin{align*}
\notag &\|\Phi(y_{t+1}+z_t,\xi_{t})-\Phi(x_t+z_t',\xi_t')\|^2\\
\notag 	& =	\|F(y_{t+1}+z_t)+w_t-F(x_t+z_t)-w_t'\|^2\\ \notag
&=\|F_t(y_{t+1})+\Delta_t+w_t-F_t(x_t) -\Delta'_t-w_t'\|^2.
\end{align*}
By the triangle inequality, we further obtain
\begin{align}\label{ineq:05}
&\|\Phi(y_{t+1}+z_t,\xi_{t})-\Phi(x_t+z_t',\xi_t')\|^2\cr
&\leq\left(\|F_t(y_{t+1})-F_t(x_t)\|+\|\Delta_t\|+\|\Delta_t'\|+\|w_t\|+\|w_t'\|\right)^2\cr
&\leq 5\|F_t(y_{t+1})-F_t(x_t)\|^2 +5B_t \cr 
& \leq5\frac{nC^2}{\e_t^2}\|y_{t+1}-x_t\|^2+5B_t,
\end{align}
where second inequality is obtained from the following relation for any $a_1,a_2,\ldots,a_m\in \Real$ and
any integer $m\ge2$:
\[(a_1+a_2+\ldots+a_m)^2\leq m\left(a_1^2+a_2^2+\ldots+ a_m^2\right),\] 
and the last inequality in (\ref{ineq:05}) is obtained using the Lipschitzian property of mapping $F_t$ from Lemma (\ref{lemma:F_k}).}
Using the upper bound we found in (\ref{ineq:05}), \an{from} inequality (\ref{ineq:04}), we obtain
 \begin{align*}
& \quad \|x_{t+1}-y\|^2\leq \|x_t-y\|^2- \|y_{t+1}-x_t\|^2\\
&+2 \g_t\Phi(y_{t+1}+z_t,\xi_t)^T(y-y_{t+1})\cr&+5nC^2\frac{\g_t^2}{\e_t^2}\|y_{t+1}-x_t\|^2+5\g_t^2B_t \\ &\leq \|x_t-y\|^2-\left(1-5nC^2\frac{\g_t^2}{\e_t^2}\right)
	\|y_{t+1}-x_t\|^2\\
	&+2 \g_t\Phi(y_{t+1}+z_t,\xi_t)^T(y-y_{t+1})+5\g_t^2B_t.
\end{align*}
From our assumption that $\g_t\leq \frac{\e_t}{\sqrt{5n}C}$, we have $5nC^2\frac{\g_t^2}{\e_t^2}\leq 1$. Therefore, from the preceding inequality we obtain
 \begin{align}\label{ineq:06}
& \quad \|x_{t+1}-y\|^2\leq \|x_t-y\|^2 \notag \\
		&+2
		\g_t\Phi(y_{t+1}+z_t,\xi_t)^T(y-y_{t+1})+5\g_t^2B_t.
\end{align}
From the definition of $w_t$ in (\ref{equ:stochastic-errors}) we obtain
 \begin{align}\label{ineq:07}
& \quad \Phi(y_{t+1}+z_t,\xi_t)^T(y-y_{t+1})\notag\\
&=F(y_{t+1}+z_t)^T(y-y_{t+1})+w_t^T(y-y_{t+1})\notag \\
&=F(y_{t+1}+z_t)^T(y-(y_{t+1}+z_t))\notag \\
 & +F(y_{t+1}+z_t)^Tz_t+w_t^T(y-y_{t+1})\notag \\ 
&\leq  F(y_{t+1}+z_t)^T(y-(y_{t+1}+z_t))\notag \\
& +\|F(y_{t+1}+z_t)\|\|z_t\|+w_t^T(y-y_{t+1})\notag \\
&\leq  F(y_{t+1}+z_t)^T(y-(y_{t+1}+z_t))\notag \\
& +\sqrt{n}C\e_t+w_t^T(y-y_{t+1}),
\end{align} 
where the first inequality is implied by the Cauchy-Schwartz inequality
and the last inequality is obtained by boundedness of the mapping $F$
over the set $X^\e$ from Assumption~\ref{assum:step_error_sub_1} and by
\an{the} boundedness of $z_t$ \an{(since
$z_t=(z_{t,1};z_{t,2};\ldots;z_{t,n})$ and $|z_{t,i}|\leq \e_t$ for
all $i$, we have $\|z_t\|\leq \sqrt{n}\e_t$). From inequalities
(\ref{ineq:06}) and (\ref{ineq:07}), we arrive at the desired relation.}
\end{proof}
\us{Unlike optimization problems where the function provides a metric
	for distinguishing solutions, there is no immediate analog in
		variational inequality problems. However, one may prescribe a
			residual function associated with a variational inequality
			problem.}
\begin{definition}[Gap function]\label{def:gap1}
Let $X \subset \mathbb{R}^n$  be a nonempty and closed set. 
Suppose that mapping $F: X\rightarrow \mathbb{R}^n$ is defined on the set $X$. We define the following gap function, $\hbox{G}: X \rightarrow \mathbb{R}^+\cup \{0\}$ to measure the accuracy of a vector $x \in X$: 
\begin{align}\label{equ:gapf}
\hbox{G}(x)= \sup_{y \in X} F(y)^T(x-y).
\end{align}
\end{definition}
\an{We note that the gap function $\hbox{G}$ is in fact also a function of the set $X$ and the map $F$, but we do not use this dependency so \fy{we} use $\hbox{G}$ instead of~$\hbox{G}_{X,F}$.}

\fy{\begin{lemma}[Properties of gap function]\label{lemma:gap-positive}
Consider Definition (\ref{def:gap1}). We have the following properties \cite{facchinei02finite}:
\begin{itemize}
\item [(a)] The gap function (\ref{equ:gapf}) is nonnegative for any $x \in X$. 
\item [(b)] Assume that the mapping $F$ is bounded over the set X
. Then, \hbox{G} is continuous at any $x \in X$.
\end{itemize}
\end{lemma}}

\an{
\begin{proposition}[Error bounds on the expected gap value]\label{prop:errorbound-averaged}
Consider problem~(\ref{def:SVI}), and let Assumptions~\ref{assum:step_error_sub_1} 
and~\ref{assump:stochastic-error} hold. Let the weighted average sequence $\{\bar y_{k}(r$)$\}$ be defined by
$$\bar y_{k+1}(r) \triangleq \frac{\sum_{t=0}^k \gamma_t^r
	 (y_{t+1}+z_k)}{\sum_{t=0}^k \gamma_t^r},\qquad\hbox{for all }k\ge0,$$
	 where $r \in \Real$ is a parameter, $\{y_{t}\}$ is
generated by algorithm~(\ref{algorithm:RSEG}), and the stepsize sequence $\{\g_t\}$ is non-increasing and 
$0<\g_t\leq \frac{\e_k}{\sqrt{5n}C}$ for all 
$t\geq 0$. Then, the following statements are valid:\\
(a)  For any $k\geq 0$, and $r \geq 1$, we have:
   \begin{align}\label{ineq:gap-general-bound1}& \quad \EXP{
	   \hbox{G}(\bar y_{k+1}(r))} \leq 
\left(\sum_{t=0}^{k}\g_{t}^r\right)^{-1}  \left( 4M^2\g_0^{r-1} \right. \\ \notag
&\left. +\sqrt{n}C\sum_{t=0}^{k}\g_t^r\e_t +(5.5\nu^2+5C^2)\sum_{t=0}^{k}\g_t^{r+1}\right).
    \end{align}
(b) For any $k\geq 0$, and $r < 1$, we have:
   \begin{align}\label{ineq:gap-general-bound2}& \quad \EXP{ \hbox{G}(\bar y_{k+1}(r))} \leq 
\left(\sum_{t=0}^{k}\g_{t}^r\right)^{-1} 
\left( \frac{4M^2}{\g_0^{1-r}}\right. \\ \notag &\left. +\frac{4M^2}{\g_k^{1-r}}+\sqrt{n}C\sum_{t=0}^{k}\g_t^r\e_t +(5.5\nu^2+5C^2)\sum_{t=0}^{k}\g_t^{r+1}\right).
    \end{align}
\end{proposition}
}
\begin{proof}
In the first part of the proof, we allow $r$ to be any real number and
we obtain a general relation. Using the general relation, we prove parts (a) and (b)
	separately.  Let us define $u_{t+1}$ as 
	\begin{align}\label{def:u_t}u_{t+1}=\Pi_X[u_t+\g_tw_t], \quad
	\hbox{for any $t\geq 0$},\end{align} where $u_0=x_0$. Adding and
	subtracting $u_t$, \eqref{lemma:rec-bound} yields 
 \begin{align}\label{ineq:08}
& \, \quad \|x_{t+1}-y\|^2\leq \|x_t-y\|^2\notag \\
	&+2 \g_tF(y_{t+1}+z_t)^T(y-(y_{t+1}+z_t)) +2\sqrt{n}C\g_t\e_t\notag \\
&+2\g_tw_t^T(u_t-y_{t+1}) +2\g_tw_t^T(y-u_t) +5\g_t^2B_t.
\end{align}
\fy{Next, we find an upper bound for the term $2\g_tw_t^T(u_t-y)$:} 
 \begin{align*}
\|u_{t+1}-y\|^2&=\|\Pi_X[u_t+\g_tw_t]-y\|^2\leq \|u_t+\g_tw_t-y\|^2 \cr 
&= \|u_t-y\|^2+2\g_tw_t^T(u_t-y)+\g_t^2\|w_t\|^2,
\end{align*}
where the second relation is implied by Lemma \ref{lemma:projProperties}(a). \fy{Thus, }
 \begin{align*}
  2\g_tw_t^T(u_t-y) & \leq \|u_t-y\|^2-\|u_{t+1}-y\|^2 +\g_t^2\|w_t\|^2.
\end{align*}
The preceding relation and (\ref{ineq:08}) imply that 
 \begin{align}\label{ineq:10}
& \quad \|x_{t+1}-y\|^2 \leq \|x_t-y\|^2 \notag \\ 
&+2
\g_tF(y_{t+1}+z_t)^T(y-(y_{t+1}+z_t)) +2\sqrt{n}C\g_t\e_t \notag \\
&+2\g_tw_t^T(u_t-y_{t+1})
+\|u_t-y\|^2-\|u_{t+1}-y\|^2\notag \\
& +5\g_t^2(B_t+0.2\|w_t\|^2).
\end{align}
By monotonicity of the mapping $F$ over the set $X^\e$ from Assumption \ref{assum:step_error_sub_1}(b) we have  \begin{align}\label{ineq:11} 
& \quad \, \notag F(y_{t+1}+z_t)^T(y-(y_{t+1}+z_t)) \\
& \leq F(y)^T(y-(y_{t+1}+z_t)).
\end{align}
\fy{From (\ref{ineq:10}) and (\ref{ineq:11})}, and rearranging the terms we obtain
 \begin{align*}
& \quad \, \g_tF(y)^T(y_{t+1}+z_t-y)\\
	&\leq 0.5\|x_t-y\|^2-0.5\|x_{t+1}-y\|^2+ 0.5\|u_t-y\|^2\\
& -0.5\|u_{t+1}-y\|^2 +\sqrt{n}C\g_t\e_t 
+\g_tw_t^T(u_t-y_{t+1})\\& +2.5\g_t^2(B_t+0.2\|w_t\|^2).
\end{align*}
Multiplying both sides of the preceding inequality by $\g_t^{r-1}$ for some constant $r \in \Real$ and $k\geq 0$, we have
 \begin{align}\label{ineq:11-2}
& \quad \g_t^rF(y)^T(y_{t+1}+z_t-y)\leq 0.5\g_t^{r-1}\|x_t-y\|^2\cr 
&-0.5\g_t^{r-1}\|x_{t+1}-y\|^2 
+0.5\g_t^{r-1}\|u_t-y\|^2\cr &-0.5\g_t^{r-1}\|u_{t+1}-y\|^2+\sqrt{n}C\g_t^r\e_t
+\g_t^rw_t^T(u_t-y_{t+1}) \cr & +2.5\g_t^{r+1}(B_t+0.2\|w_t\|^2).
\end{align}

\noindent \textbf{(a) $\mathbf{r\geq 1}$:}
Since $\{\g_t\}$ is a non-increasing sequence and $r\geq 1$, we get $\g_{t+1}^{r-1} \leq \g_t^{r-1}$. Therefore, from relation (\ref{ineq:11-2})
 \begin{align*}
& \quad \, \g_t^rF(y)^T(y_{t+1}+z_t-y)\\
&\leq 0.5\g_t^{r-1}\|x_t-y\|^2-0.5\g_{t+1}^{r-1}\|x_{t+1}-y\|^2 \cr & +0.5\g_t^{r-1}\|u_t-y\|^2-0.5\g_{t+1}^{r-1}\|u_{t+1}-y\|^2+\sqrt{n}C\g_t^r\e_t \cr &+\g_t^rw_t^T(u_t-y_{t+1}) +2.5\g_t^{r+1}(B_t+0.2\|w_t\|^2).
\end{align*}
Summing over $t$ from $t=0$ to $k$, we obtain
 \begin{align*}
& \quad \, \sum_{t=0}^{k}\g_t^rF(y)^T(y_{t+1}+z_t-y)\\
	&\leq \frac{\g_0^{r-1}}{2}\|x_0-y\|^2-\frac{\g_{k+1}^{r-1}}{2}\|x_{k+1}-y\|^2 \cr
	&
	+\frac{\g_0^{r-1}}{2}\|u_0-y\|^2-\frac{\g_{k+1}^{r-1}}{2}\|u_{k+1}-y\|^2+\sqrt{n}C\sum_{t=0}^{k}\g_t^r\e_t
	\cr &+\sum_{t=0}^{k}\g_t^rw_t^T(u_t-y_{t+1})  +2.5\sum_{t=0}^{k}\g_t^{r+1}(B_t+0.2\|w_t\|^2).
\end{align*}
From boundedness of the set $X$ and the triangle inequality, we have $\|u_0-y\|^2=\|x_0-y\|^2\leq 4M^2$. Thus, from the definition of $\bar y_{k+1}(r)$\fy{, the preceding relation implies that }
 \begin{align*}
& \quad \left(\sum_{t=0}^{k}\g_{t}^r\right)F(y)^T(\bar y_{k+1}(r)-y)\cr 
&\leq 4M^2\g_0^{r-1}+\sqrt{n}C\sum_{t=0}^{k}\g_t^r\e_t +\sum_{t=0}^{k}\g_t^rw_t^T(u_t-y_{t+1}) \cr &+2.5\sum_{t=0}^{k}\g_t^{r+1}(B_t+0.2\|w_t\|^2).
\end{align*}
Taking supremum over the set $X$ with respect to $y$, invoking the definition of the gap function (\ref{def:gap1}),\fy{ we have}
 \begin{align*}
& \quad \left(\sum_{t=0}^{k}\g_{t}^r\right)\hbox{G}(\bar
		y_{k+1}(r))\leq 4M^2\g_0^{r-1}+\sqrt{n}C\sum_{t=0}^{k}\g_t^r\e_t\cr & +\sum_{t=0}^{k}\g_t^rw_t^T(u_t-y_{t+1}) +2.5\sum_{t=0}^{k}\g_t^{r+1}(B_t+0.2\|w_t\|^2).
\end{align*}
Taking expectations, we obtain the following:
 \begin{align}\label{ineq:12}
& \quad \left(\sum_{t=0}^{k}\g_{t}^r\right)\EXP{\hbox{G}(\bar
		y_{k+1}(r))}\cr &\leq 4M^2\g_0^{r-1}+\sqrt{n}C\sum_{t=0}^{k}\g_t^r\e_t +\sum_{t=0}^{k}\g_t^r\EXP{w_t^T(u_t-y_{t+1})} \cr &+2.5\sum_{t=0}^{k}\g_t^{r+1}\EXP{B_t+0.2\|w_t\|^2}.
\end{align}

The algorithm (\ref{algorithm:RSEG}) \fy{implies that} $y_{t+1}$ is $\left(\sF_{t-1}\cup \sF_{t}'\right)$-measurable and $x_{t+1}$ is $\left(\sF_{t}\cup \sF_{t}'\right)$-measurable. Moreover, the definition of $w_t$ in (\ref{equ:stochastic-errors}), and the definition of $u_t$ in (\ref{def:u_t}) imply that $w_t$ is $\left(\sF_{t}\cup \sF_{t}'\right)$-measurable and $u_t$ is $\left(\sF_{t-1}\cup \sF_{t-1}'\right)$-measurable. Thus, the term $u_t-y_{t+1}$ is $\left(\sF_{t-1}\cup \sF_{t}'\right)$-measurable. Also, from Assumption \ref{assump:stochastic-error}, $F(y_{t+1}+z_t)=\EXP{\Phi(y_{t+1}+z_t,\xi_t)\mid \sF_{t-1}\cup \sF_{t}'\cup \{z_t\}}$. Therefore, 
\begin{align*}
& \quad \EXP{w_t^T(u_t-y_{t+1})\mid \sF_{t-1}\cup \sF_{t}'\cup\{z_t\}}\cr & =(u_t-y_{t+1})^T\EXP{w_t  \mid \sF_{t-1}\cup \sF_{t}'\cup\{z_t\}}\cr 
& =(u_t-y_{t+1})^T\EXP{\Phi(y_{t+1}+z_t,\xi_t) \mid \sF_{t-1}\cup \sF_{t}'\cup\{z_t\}} \cr 
& -(u_t-y_{t+1})^T F(y_{t+1}+z_t)=0,
\end{align*}
where \fy{we use the} unbiasedness of the mapping $F$. 
Taking expectation on the preceding equation, we obtain
\begin{align}\label{ineq:13}
\EXP{w_t^T(u_t-y_{t+1})}=0, \quad \hbox{for any } t \geq 0.
\end{align}	
Next, we estimate $\EXP{\|w_t\|^2}$. \fy{From Assumption \ref{assump:stochastic-error}, we have}
\begin{align*}
\EXP{\|w_t\|^2\mid \sF_{t-1}\cup \sF_{t}'\cup \{z_t\}} &\leq \nu^2,\quad \hbox{for any } t \geq 0.
\end{align*}
 Taking expectations, we obtain the following inequality: 
 \begin{align}\label{ineq:13-1}
& \quad \EXP{\EXP{\|w_t\|^2\mid \sF_{t-1}\cup \sF_{t}'\cup \{z_t\}}} \leq \nu^2
\cr & \Rightarrow \EXP{\|w_t\|^2} \leq \nu^2, \quad \hbox{for any } t \geq 0.
\end{align}
In a similar fashion, we can show the following: 
\begin{align}\label{ineq:13-2}
\EXP{\|w_t'\|^2} \leq \nu^2, \quad \hbox{for any } t \geq 0.
\end{align}
Next, we estimate $\EXP{\|\Delta_t\|^2}$. From the definition of $\Delta_t$ in (\ref{equ:stochastic-errors}): 
\begin{align}\label{ineq:14}
& \quad \EXP{\|\Delta_t\|^2\mid \sF_{t-1}\cup \sF_{t}'}\cr &= \EXP{\|F(y_{t+1}+z_t)-F_k(y_{t+1})\|^2\mid \sF_{t-1}\cup \sF_{t}'}\cr &= \EXP{\|F(y_{t+1}+z_t)\|^2\mid \sF_{t-1}\cup \sF_{t}'}\notag \\ &+ \EXP{\|F_k(y_{t+1})\|^2\mid \sF_{t-1}\cup \sF_{t}'}\cr &-2 \EXP{F(y_{t+1}+z_t)^TF_t(y_{t+1})\mid \sF_{t-1}\cup \sF_{t}'}.
\end{align}
Note that from \fy{Definition \ref{def:F_k}}, $F_t(y_{t+1})=\EXP{F(y_{t+1}+z_k)\mid \sF_{t-1}\cup \sF_{t}'}$. Since $y_{t+1}$ is $\left(\sF_{t-1}\cup \sF_{t}'\right)$-measurable, we observe that $F_t(y_{t+1})$ is also $\left(\sF_{t-1}\cup \sF_{t}'\right)$-measurable. Therefore, from relation (\ref{ineq:14}) we have
\begin{align*}
& \quad \EXP{\|\Delta_t\|^2\mid \sF_{t-1}\cup \sF_{t}'} \cr &= \EXP{\|F(y_{t+1}+z_t)\|^2\mid \sF_{t-1}\cup \sF_{t}'}+ \|F_t(y_{t+1})\|^2\cr &-2 \EXP{F(y_{t+1}+z_t)\mid \sF_{t-1}\cup \sF_{t}'}^TF_k(y_{t+1})\cr 
& = \EXP{\|F(y_{t+1}+z_t)\|^2\mid \sF_{t-1}\cup \sF_{t}'}+
\|F_k(y_{t+1})\|^2\\
	& -2 F_t(y_{t+1})^TF_t(y_{t+1})\cr 
& \leq \EXP{\|F(y_{t+1}+z_t)\|^2\mid \sF_{t-1}\cup \sF_{t}'} \leq 
C^2,
\end{align*}
where the last inequality is \fy{obtained using boundedness} of
the mapping $F$ over the set $X^\e$. Taking expectations over the preceding relation, we get  
\begin{align}\label{ineq:15}
& \quad \EXP{\EXP{\|\Delta_t\|^2\mid \sF_{t-1}\cup \sF_{t}'}} \leq C^2
\cr & \Rightarrow \EXP{\|\Delta_t\|^2} \leq C^2, \quad \hbox{for any } t \geq 0.
\end{align}
In a similar fashion, we can show that 
\begin{align}\label{ineq:16}
\EXP{\|\Delta_t'\|^2} \leq C^2, \quad \hbox{for any } t \geq 0.
\end{align}
In conclusion, invoking relations (\ref{ineq:13}), (\ref{ineq:13-1}),
   (\ref{ineq:13-2}), (\ref{ineq:15}), and (\ref{ineq:16}), from
   relation (\ref{ineq:12}) we conclude  with the following:
\begin{align*}
& \quad \left(\sum_{t=0}^{k}\g_{t}^r\right)\EXP{\hbox{G}(\bar
		y_{k+1}(r))}\cr &\leq 4M^2\g_0^{r-1}+\sqrt{n}C\sum_{t=0}^{k}\g_t^r\e_t +(5.5\nu^2+5C^2)\sum_{t=0}^{k}\g_t^{r+1}
\end{align*}
implying the desired result (\ref{ineq:gap-general-bound1}).

\noindent \textbf{(b) $\mathbf{r<1}$:} Consider relation (\ref{ineq:11}). Adding and subtracting the terms $0.5\g_{t-1}^{r-1}\|x_t-y\|^2$ and $0.5\g_{t-1}^{r-1}\|u_t-y\|^2$, we can write
 \begin{align*}
& \quad \, \g_t^rF(y)^T(y_{t+1}+z_t-y) \cr 
&\leq \frac{\|x_t-y\|^2}{2\g_{t-1}^{1-r}}-\frac{\|x_{t+1}-y\|^2}{2\g_{t}^{1-r}} + \frac{\|u_t-y\|^2}{2\g_{t-1}^{1-r}}-\frac{\|u_{t+1}-y\|^2}{2\g_{t}^{1-r}}\cr 
&+\left(\frac{1}{2\g_{t}^{1-r}}-\frac{1}{2\g_{t-1}^{1-r}}\right)\left(\|x_t-y\|^2+\|u_t-y\|^2\right)\cr &+\sqrt{n}C\g_t^r\e_t +\g_t^rw_t^T(u_t-y_{t+1})\cr &+2.5\g_t^{r+1}\left(B_t+0.2\|w_t\|^2\right).
\end{align*}
Boundedness of $X$ implies that $\|x_t-y\|^2+\|u_t-y\|^2 \leq 8M^2$. \fy{By summing over $t$ from $t=1$ to} $k$, we obtain
 \begin{align*}
& \quad \sum_{t=1}^{k}\g_t^rF(y)^T(y_{t+1}+z_t-y) \cr &\leq  \frac{\|x_1-y\|^2}{2\g_{0}^{1-r}}-\frac{\|x_{k+1}-y\|^2}{2\g_{k}^{1-r}} + \frac{\|u_1-y\|^2}{2\g_{0}^{1-r}}-\frac{\|u_{k+1}-y\|^2}{2\g_{k}^{1-r}}\cr 
&+4M^2\left(\frac{1}{\g_{k}^{1-r}}-\frac{1}{\g_{0}^{1-r}}\right)+\sqrt{n}C\sum_{t=1}^{k}\g_t^r\e_t \cr &+\sum_{t=1}^{k}\g_t^rw_t^T(u_t-y_{t+1}) +2.5\sum_{t=1}^{k}\g_t^{r+1}\left(B_t+0.2\|w_t\|^2\right).
\end{align*}
By letting $t=0$ in relation (\ref{ineq:11}), and then adding the
resulting inequality to the preceding inequality, we obtain the
following:
 \begin{align*}
& \quad \sum_{t=0}^{k}\g_t^rF(y)^T(y_{t+1}+z_t-y) \cr &\leq  \frac{\|x_0-y\|^2}{2\g_{0}^{1-r}}-\frac{\|x_{k+1}-y\|^2}{2\g_{k}^{1-r}} + \frac{\|u_0-y\|^2}{2\g_{0}^{1-r}}-\frac{\|u_{k+1}-y\|^2}{2\g_{k}^{1-r}}\cr 
&+4M^2\left(\frac{1}{\g_{k}^{1-r}}-\frac{1}{\g_{0}^{1-r}}\right)+\sqrt{n}C\sum_{t=0}^{k}\g_t^r\e_t  \cr &+\sum_{t=0}^{k}\g_t^rw_t^T(u_t-y_{t+1})+2.5\sum_{t=0}^{k}\g_t^{r+1}\left(B_t+0.2\|w_t\|^2\right).
\end{align*}
Invoking boundedness of the set $X$ again, we obtain
 \begin{align*}
& \quad \sum_{t=0}^{k}\g_t^rF(y)^T(y_{t+1}+z_t-y) \cr &\leq  \frac{4M^2}{\g_{0}^{1-r}}+\frac{4M^2}{\g_{k}^{1-r}}+\sqrt{n}C\sum_{t=0}^{k}\g_t^r\e_t +\sum_{t=0}^{k}\g_t^rw_t^T(u_t-y_{t+1}) \cr &+2.5\sum_{t=0}^{k}\g_t^{r+1}\left(B_t+0.2\|w_t\|^2\right).
\end{align*}
The remainder of the proof can be done in a similar fashion to the proof of part (a).
\end{proof}
\begin{theorem}[Optimal rate of convergence for $\bar y_{k}(r)$]\label{prop:rate-averaging}
\an{Under assumptions of Proposition~\ref{prop:errorbound-averaged}, consider
the weighted average sequence $\{\bar y_{k}(r)\}$ of the sequence $\{y_t\}$
generated by algorithm (\ref{algorithm:RSEG}), where 
$$\g_t=\frac{\g_0}{\sqrt{t+1}}, \quad \e_t=\frac{\e_0}{\sqrt{t+1}}, \quad \hbox{ with }
\g_0\leq \frac{\e_0}{\sqrt{5n}C}.$$ }
 Then, when $r<1$, we have 
 $$\EXP{\hbox{G}(\bar y_{k}(r))} \leq \frac{\theta_r}{\sqrt{k}}, \quad \hbox{for all }k \geq 1,$$
 where $\theta_r \triangleq 4(2-r)\g_0^{-r}M^2+\frac{2-r}{1-r}(\sqrt{n}C\e_0+\g_0(5.5\nu^2+5C^2)).$
\end{theorem}
\begin{proof}
Let us define 
 \begin{align*}
\hbox{Term A}\triangleq \frac{(k+1)^{\frac{1-r}{2}}}{\sum_{t=0}^{k}(t+1)^{-\frac{r}{2}}}, \hbox{ Term B}\triangleq \frac{\sum_{t=0}^{k}(t+1)^{-\frac{(1+r)}{2}}}{\sum_{t=0}^{k}(t+1)^{-\frac{r}{2}}}. 
\end{align*}
Consider Proposition \ref{prop:errorbound-averaged}(b). Note that $\g_k \leq \g_0$ and $r<1$ imply that $\frac{4M^2}{\g_0^{1-r}} \leq \frac{4M^2}{\g_k^{1-r}}$. From the definitions of Term A and B, the relation Proposition \ref{prop:errorbound-averaged}(b) implies that
\begin{align}\label{ineq:gap-termAB}
\EXP{\hbox{G}(\bar y_{k+1}(r))}\leq R_1(\hbox{Term A})+R_2(\hbox{Term B}),
\end{align}
where $R_1\triangleq 8\g_0^{-r}M^2$ and $R_2\triangleq \sqrt{n}C\e_0+\g_0(5.5\nu^2+5C^2)$.
We make use of the following relation in our analysis:
\begin{align*}
 \int_{0}^{k+1} (x+1)^{-p}dx \leq \sum_{t=0}^{k}(t+1)^{-p} \leq 1+ \int_{0}^{k} (x+1)^{-p} dx,
\end{align*} where $p \in \Real$ and $k \geq 0$. From this relation, it follows that 
\begin{align*}\hbox{Term A}& \leq \frac{(k+1)^{\frac{1-r}{2}}}{ \int_{0}^{k+1} (x+1)^{-\frac{r}{2}}dx}=\left(1-\frac{r}{2}\right)\frac{(k+1)^{\frac{1-r}{2}}}{ (k+1)^{\left(1-\frac{r}{2}\right)}}\cr & =\frac{2-r}{2\sqrt{k+1}},\end{align*}
\begin{align*}\hbox{Term B}& \leq \frac{\int_{0}^{k} (x+1)^{-\frac{(1+r)}{2}}dx}{ \int_{0}^{k+1} (x+1)^{-\frac{r}{2}}dx}=\left(\frac{2-r}{1-r}\right)\frac{(k+1)^{\frac{1-r}{2}}}{ (k+1)^{\left(1-\frac{r}{2}\right)}}\cr & =\frac{2-r}{(1-r)\sqrt{k+1}}.\end{align*}
In conclusion, replacing the preceding bounds for Term $A$ and Term $B$ in relation (\ref{ineq:gap-termAB}), we get the desired result.
\end{proof}

\section{\fy{Recursive Stepsize and Smoothing}}\label{sec:algorithmII}
\fy{Motivated by the little guidance on the choice of a diminishing stepsize, in this section, we consider algorithm (\ref{algorithm:RSEG})}
\fy{and assume that the} stepsize $\g_t$ and the smoothing sequence $\e_t$ are given by 
\begin{equation}\label{equ:adaptive-rules}
  \begin{split}
\left\{ 
  \begin{array}{l}
   \g_0^*= \frac{2\alpha M}{q}\\
\g_{t+1}^*= \g_t^*\left(1-\frac{\alpha}{2M}\g_t^*\right)\end{array}\right.
  \end{split}
,\hbox{ } 
  \begin{split}
\left\{ 
  \begin{array}{l}
   \e_0^*= \frac{2\alpha \beta M}{q}\\
\e_{t+1}^*= \e_t^*\left(1-\frac{\alpha}{2\beta M}\e_t^*\right) \end{array}\right.,
  \end{split}
\end{equation}
where $\beta$ is a constant such that $ \beta > \max\left\{\sqrt{5n}C,\frac{\alpha^2-5(C^2+\nu^2)}{2\sqrt{n}C}\right\}$ and  $q\triangleq \alpha+2C^2\alpha+2\alpha \nu^2+2(\alpha+1)\beta\sqrt{n}C+5(C^2+\nu^2).$
Our goal is to analyze the convergence of $\{x_k\}$ to the solution of problem (\ref{def:SVI}).
\begin{definition}[Weak-sharpness property]\label{def:weak-sharpness} Consider VI$(X,F)$ where $X \subset \Real^n$ and $F:X\rightarrow \Real^n$ is a continuous mapping. Let $X^*$ denote the solution set of VI$(X,F)$. The problem has a weak-sharpness property with parameter $\alpha>0$, if for all $x \in X$ and all $x^* \in X^*$
 \begin{align}\label{ineq:weak-sharp}
F(x^*)^T(x-x^*)\geq \alpha \hbox{dist}(x,X^*).
\end{align}
\end{definition}
\begin{lemma}[A recursive error bound]\label{lemma:bound-sharpness}
Consider algorithm (\ref{algorithm:RSEG}). Let Assumption \ref{assum:step_error_sub_1} and \ref{assump:stochastic-error} hold and suppose mapping $F$ is strictly monotone over $X$ and $\g_k \leq \frac{\e_k}{\sqrt{5n}C}$  for any $k\geq 0$. Moreover, assume that problem (\ref{def:SVI}) has the weak-sharpness property with parameter $\alpha>0$. Then,  problem (\ref{def:SVI}) has a unique solution, $x^*$, and the following relation holds:
 \begin{align}\label{ineq:bound-sharpness}
&\quad \EXP{\|x_{t+1}-x^*\|^2\mid \sF_{t-1}\cup \sF'_t} \cr &\leq \left(1-\frac{\alpha \g_t}{M}\right)\|x_t-x^*\|^2+q_1\e_t\g_t+q_2\g_t^2,
\end{align}
where $q_1\triangleq \alpha(1+2(C^2+ \nu^2))+5(C^2+\nu^2)$ and $q_2\triangleq 2(1+\alpha)\beta\sqrt{n}C$.
\end{lemma}
\begin{proof}
Since the mapping $F$ is strictly monotone over the closed and convex set $X$, the problem VI$(X,F)$ has at most one solution. From non-emptiness of the solution set of the problem (\ref{def:SVI}), we conclude that it has a unique solution. Let $x^*$ be such a solution.
From relation (\ref{ineq:rec-bound}), for $y=x^*$ and the monotonicity property of the mapping $F$ we have
 \begin{align*}
\|x_{t+1}-x^*\|^2 &\leq \|x_t-x^*\|^2+2 \g_tF(x^*)^T(x^*-(y_{t+1}+z_t))\cr 
&+2\sqrt{n}C\g_t\e_t+2\g_tw_t^T(x^*-y_{t+1})+5\g_t^2B_t.
\end{align*}
Invoking the weak-sharpness property and the uniqueness of the solution set, \fy{we obtain}
  \begin{align*}
\|x_{t+1}-x^*\|^2 &\leq \|x_t-x^*\|^2-2\alpha \g_t\|y_{t+1}+z_t-x^*\|\cr 
&+2\sqrt{n}C\g_t\e_t+2\g_tw_t^T(x^*-y_{t+1})+5\g_t^2B_t.
\end{align*}
\fy{Taking conditional expectation, we obtain}
  \begin{align}\label{ineq:rec-bound1}
&\quad \EXP{\|x_{t+1}-x^*\|^2\mid \sF_{t-1}\cup \sF'_t}\cr & \leq \|x_t-x^*\|^2-2\alpha \g_t\EXP{\|y_{t+1}+z_t-x^*\|\mid \sF_{t-1}\cup \sF'_t}\cr &+2\sqrt{n}C\g_t\e_t+5\g_t^2\left(C^2+\nu^2 \right),
\end{align}
where we used $\EXP{w_t^T(x^*-y_{t+1})\mid \sF_{t-1}\cup \sF'_t}=0$ . Using the triangle inequality, we can write\begin{align}\label{ineq:rec-bound2}
&\|y_{t+1}+z_t-x^*\|\geq \|y_{t+1}-x^*\|-\|z_t\| \cr
\Rightarrow & -2\alpha \g_t\|y_{t+1}+z_t-x^*\| \leq -2\alpha \g_t\|y_{t+1}-x^*\|\cr
 &+2\alpha \g_t\|z_t\| \leq-2\alpha \g_t\|y_{t+1}-x^*\|+2\alpha\sqrt{n}C\g_t\e_t.
\end{align}
Next, we find a lower bound for the term $\|y_{t+1}-x^*\|$. Using the triangle inequality we have
\begin{align}\label{ineq:rec-bound3}
\|y_{t+1}-x^*\| &\geq \|x_t-x^*\| - \|y_{t+1}-x_t\|  \cr 
& \geq \frac{\|x_t-x^*\|^2}{2M}- \|y_{t+1}-x_t\|,
\end{align}
where in the last inequality we used the boundedness of the set $X$. We also have
\begin{align}\label{ineq:rec-bound4}
 &\quad\|y_{t+1}-x_t\|=\|\Pi_{X}\left[x_t-\g_t\Phi(x_t+z_t',\xi_t')\right]-x_t\|\cr 
 &\leq \|x_t-\g_t\Phi(x_t+z_t',\xi_t')-x_t\| = \g_t\|\Phi(x_t+z_t',\xi_t')\| \cr &=\g_t\|F(x_t+z_t')+w'_t\|\leq  0.5\g_k\left(1^2+\|F(x_t+z_t')+w'_t\|^2\right) \cr &\leq 0.5\g_k\left(1^2+2\|F(x_t+z_t')\|^2+2\|w'_t\|^2\right)\cr &\leq 0.5\g_k\left(1+2C^2+2\|w'_t\|^2\right).
\end{align}
Relations (\ref{ineq:rec-bound2}), (\ref{ineq:rec-bound3}), and (\ref{ineq:rec-bound4}) imply that
\begin{align*}
 &-2\alpha \g_t\|y_{t+1}+z_t-x^*\| \leq -\frac{\alpha \g_t}{M}\|x_t-x^*\|^2 +\alpha\g_t^2\cr &+2C^2\alpha\g_t^2+2\alpha\g_t^2\|w'_t\|^2 + 2\alpha\sqrt{n}C\g_t\e_t.
\end{align*}
From the preceding relation, and (\ref{ineq:rec-bound1}) we have
  \begin{align*}
&\quad \EXP{\|x_{t+1}-x^*\|^2 \mid\sF_{t-1}\cup \sF'_t}\cr & \leq \left(1-\frac{\alpha \g_t}{M}\right)\|x_t-x^*\|^2+\alpha\g_t^2+2C^2\alpha\g_t^2+2\alpha\g_t^2\nu^2\cr &+ 2\alpha\sqrt{n}C\g_t\e_t+2\sqrt{n}C\g_t\e_t+5\g_t^2\left(C^2+\nu^2 \right).
\end{align*}
Replacing $\e_t$ by $\beta\g_t$, we obtain the desired relation.
\end{proof}
We use the following Lemma in establishing the almost-sure convergence \fy{(cf.~Lemma 10, page 49~\cite{Polyak87})}.  
\begin{lemma}\label{lemma:probabilistic_bound_polyak}
Let $\{v_k\}$ be a sequence of nonnegative random variables, where 
$\EXP{v_0} < \infty$, and let $\{\a_k\}$ and $\{\mu_k\}$
be deterministic scalar sequences such that:
\begin{align*}
& \EXP{v_{k+1}|v_0,\ldots, v_k} \leq (1-\alpha_k)v_k+\mu_k
\qquad a.s. \ \hbox{for all }k\ge0, \cr
& 0 \leq \alpha_k \leq 1, \quad\ \mu_k \geq 0, \cr
& \quad\ \sum_{k=0}^\infty \alpha_k =\infty, 
\quad\ \sum_{k=0}^\infty \mu_k < \infty, 
\quad\ \lim_{k\to\infty}\,\frac{\mu_k}{\alpha_k} = 0. 
\end{align*}
Then, $v_k \rightarrow 0$ almost surely. 
\end{lemma} 
\begin{theorem}[Optimal rate of convergence for $x_k$]
Consider algorithm (\ref{algorithm:RSEG}). Let Assumption \ref{assum:step_error_sub_1} and \ref{assump:stochastic-error} hold, 
 let mapping $F$ be strictly monotone over $X$, and assume that problem (\ref{def:SVI})  has the weak-sharpness property with parameter $\alpha>0$. Suppose the sequences $\g_t$ and $\e_t$ are given by the recursive relations (\ref{equ:adaptive-rules}). Then, problem (\ref{def:SVI}) has a unique solution, $x^*$, and the following results hold:
\begin{itemize}
\item [(a)]
The sequence $\{x_t\}$ generated by the algorithm (\ref{algorithm:RSEG}) converges to the solution of problem (\ref{def:SVI}) almost surely as $k\rightarrow \infty$.
\item [(b)] The sequence $\{x_t\}$ generated by the algorithm (\ref{algorithm:RSEG}) converges to the solution of problem (\ref{def:SVI}) in a mean-squared sense. More precisely, we have 
\[\EXP{\|x_{t}-x^*\|^2}\leq \left(\frac{4qM^2}{\alpha^2}\right)\frac{1}{t},\quad \hbox{for all } \fy{t\geq 1}.\]
\end{itemize}
\end{theorem}
\begin{proof} 
(a) First we show that $\e_t^*=\beta\g_t^*$ for any $t \geq 0$. From (\ref{equ:adaptive-rules}), we have $ \e_0^*= \frac{2\alpha \beta M}{q}=\beta \g_0^*$, implying that the relation holds for $t=0$. Assume that the relation holds for some fixed $t$. We show that it holds for $t+1$. We have 
\begin{align*}\e_{t+1}^*&= \e_t^*\left(1-\frac{\alpha}{2\beta M}\e_t^*\right)=\beta \g_t^*\left(1-\frac{\alpha}{2\beta M}\beta\g_t^*\right)\cr &=\beta \g_{t+1}^*\end{align*}
Therefore, we conclude that $\e_t^*=\beta\g_t^*$ for any $t \geq 0$. Since we assumed $\beta> \sqrt{5n}C$, we get $\e_t^>\sqrt{5n}C\g_t^*$ or equivalently, $\g_t < \frac{\e_t}{ \sqrt{5n}C}$ for any $t \geq 0$. Next, we show that $\{\g_t^*\}$ is a decreasing sequence with strictly positive elements. We have \begin{align*}
&\frac{\alpha^2-5(C^2+\nu^2)}{2\sqrt{n}C} <\beta \Rightarrow \alpha^2 < 5(C^2+\nu^2)+2\sqrt{n}C\beta \cr
& \Rightarrow \alpha^2 < \alpha(1+2C^2+2\nu^2+2\sqrt{n}C\beta)+5(C^2+\nu^2)\cr &+2\sqrt{n}C\beta =q
 \Rightarrow \frac{\alpha}{q} < \frac{1}{\alpha} 
  \Rightarrow  \g_0^* < \frac{2M}{\alpha}.
\end{align*}
Using the preceding relation we obtain $\g^*_1=\g^*_0(1-\frac{\alpha}{2M}\g^*_0)<\g_0^*<  \frac{2M}{\alpha}$ and $\g_1^*>0$. Following the same approach, induction implies that 
\begin{align}\label{ineq:rec-dec-pos}0<\ldots<\g_2^*<\g_1^*<\g_0^*<\frac{2M}{\alpha},\end{align} 
\fy{verifying that $\{\g_t^*\}$ is a decreasing sequence with strictly positive terms. Therefore, all the conditions of Lemma \ref{lemma:bound-sharpness} hold showing that for the unique solution $x^*$, 
we have}
 \begin{align*}
&\EXP{\|x_{t+1}-x^*\|^2\mid  \sF_{t-1}\cup \sF'_t} \cr &\leq \left(1-\frac{\alpha \g_t}{M}\right)\|x_t-x^*\|^2+q_1\e_t\g_t+q_2\g_t^2 \cr 
& \leq  \left(1-\frac{\alpha \g_t}{M}\right)\|x_t-x^*\|^2+\beta q_1\g_t^2+q_2\g_t^2 ,
\end{align*}
where we used the relation $\e_t^*=\beta\g_t^*$. From definition of $q_1$, $q_2$ and $q$, we have $q = \beta q_1+q_2$. Thus, from the preceding relation we obtain
 \begin{align*}
\EXP{\|x_{t+1}-x^*\|^2 \mid\sF_{t-1}\cup \sF'_t} &\leq  \left(1-\frac{\alpha \g_t}{M}\right)\|x_t-x^*\|^2\cr &+q\g_t^2, \quad \hbox{for all } t \geq 0.
\end{align*}
\fy{The next step is to show $\sum_{t=0}^\infty{\g_t^*}=\infty$ and $\sum_{t=0}^\infty{(\g_t^*)^2}<\infty.$ The proof of these relations can be found in our prior work (cf. Prop. 3 in \cite{Farzad1})}. 
The last step of the proof is applying Lemma \ref{lemma:probabilistic_bound_polyak} on the preceding inequality. We verified that all the conditions of Lemma \ref{lemma:probabilistic_bound_polyak} are satisfied for $v_t\triangleq \|x_t-x^*\|^2$, $\alpha_t\triangleq \frac{\alpha}{2M}\g_t^*$, $\mu_t\triangleq q(\g_t^*)^2$. Therefore, $x_t$ converges to $x^*$ almost surely.\\
(b) In the first part of the proof, using induction on $t$, we show that using the sequences $\g_t^*$ and $\e_t^*$, we have \begin{align}\label{ineq:bound-gamma^*_1}
\EXP{\|x_{t}-x^*\|^2}\leq \frac{2Mq}{\alpha}\g_t^*, \quad \hbox{for all } t\geq 0.\end{align} For $t=0$, the relation becomes $\EXP{\|x_{0}-x^*\|^2}\leq \frac{2Mq}{\alpha}\g_0^*=4M^2$. This holds because $\EXP{\|x_{0}-x^*\|^2} \leq \EXP{2\|x_{0}\|^2+2\|x^*\|^2}\leq 4M^2.$ Let us assume that relation (\ref{ineq:bound-gamma^*_1}) holds for $t$. Then, taking expectation from (\ref{ineq:bound-sharpness}) and the definition of $q$ we can write   
 \begin{align*} 
\EXP{\|x_{t+1}-x^*\|^2 }&\leq  \left(1-\frac{\alpha \g_t^*}{M}\right)\EXP{\|x_t-x^*\|^2}\cr &+q(\g_t^*)^2, \quad \hbox{for all } t \geq 0.
\end{align*}
Using the induction hypothesis, from the preceding relation it follows
 \begin{align*}
&\quad \EXP{\|x_{t+1}-x^*\|^2 }\leq  \left(1-\frac{\alpha \g_t^*}{M}\right)\frac{2Mq}{\alpha}\g_t^*+q(\g_t^*)^2\cr &=\frac{2Mq}{\alpha}\g_t^*\left(1-\frac{\alpha \g_t^*}{M} +\frac{\alpha \g_t^*}{2M}\right)=\frac{2Mq}{\alpha}\g_{t+1}^*.
\end{align*}
Therefore, the relation (\ref{ineq:bound-gamma^*_1}) holds for $t+1$, implying that it holds for any $t\geq 0$. In the second part of the proof, we show that 
\begin{align}\label{ineq:bound-gamma^*_2}
\g_t^*\leq \frac{2M}{\alpha t}, \quad \hbox{for all } t\geq 0.
\end{align}
From  definition of the sequence $\{\g_t^*\}$ we have for $t \geq 0$, $$\frac{1}{\g_{t+1}^*}=\frac{1}{\g_t^*(1-\frac{\alpha\g_t^*}{2M})}=\frac{1}{\g_t^*}+\frac{\frac{\alpha}{2M}}{1-\frac{\alpha\g_t^*}{2M}} .$$
Summing up from $t=0$ to $k$ we obtain
\begin{align}\label{ineq:rate_analysis_1}
 \frac{1}{\g_{k+1}^*}=\frac{1}{\g_0^*}+\frac{\alpha}{2M}\sum_{t=0}^k\frac{1}{1-\frac{\alpha\g_t^*}{2M}}>\frac{\alpha}{2M}\sum_{t=0}^k\frac{1}{1-\frac{\alpha\g_t^*}{2M}}. \end{align}
From the \fy{Cauchy-Schwarz} inequality, 
$ \frac{1}{\frac{1}{n}\sum_{k=1}^{n} \frac{1}{a_k}}\leq \frac{1}{n}\sum_{k=1}^{n} a_k$ holds for arbitrary positive numbers $a_1,a_2, \ldots,a_n$. Thus, for the terms $1-\frac{\alpha\g_t^*}{2M}$ we get
\begin{align*}
&\quad \left(\frac{1}{k+1}\sum_{t=0}^{k} \frac{1}{1-\frac{\alpha\g_t^*}{2M}}\right)^{-1}\leq \frac{1}{k+1}\sum_{t=0}^k (1-\frac{\alpha\g_t^*}{2M}) \cr &< \frac{\sum_{t=0}^{k} 1}{k+1} = 1 \Rightarrow \sum_{t=0}^{k} \frac{1}{1-\frac{\alpha\g_t^*}{2M}}>k+1.
\end{align*}
The preceding relation and (\ref{ineq:rate_analysis_1}) imply that the inequality (\ref{ineq:bound-gamma^*_2}) holds. In conclusion, using the two relations (\ref{ineq:bound-gamma^*_1}) and (\ref{ineq:bound-gamma^*_2}), we obtain the desired result.
\end{proof}

\section{Concluding Remarks}\label{sec:conc}
\fy{We presented two robust variants of a stochastic extragradient method for solving monotone stochastic VIs by utilizing a smoothing technique. First, using a new class of choices for the weights in an averaging  scheme, we show that a suitably defined gap function converges to zero at rate of ${\cal O}\left(\frac{1}{\sqrt{k}}\right)$. Second, we develop a recursive rule for updating stepsize and smoothing parameters and we show both the almost-sure convergence and that the rate in mean-squared sense is optimal. Importantly, this scheme allows for tuning the steplength sequence to problem parameters.} 
\bibliographystyle{IEEEtran}
\bibliography{wsc11-v02,demobib}

\end{document}